\documentclass[12pt,reqno]{amsart}
\usepackage{amsmath, amssymb, amsfonts, amsthm, xcolor}
\usepackage{mathtools}
\usepackage{hyperref}
\newtheorem{thm}{Theorem}[section]
\newtheorem{lem}[thm]{Lemma}
\newtheorem{prop}[thm]{Proposition}
\theoremstyle{definition}
\newtheorem{dfn}[thm]{Definition}
\newtheorem{exple}[thm]{Example}

\newtheorem{remark}[thm]{Remark}
\theoremstyle{plain}
\newtheorem{cor}[thm]{Corollary}

\numberwithin{equation}{section}
\usepackage[a4paper, top = 1.23in, bottom = 1.3in, left = 1.3in, right = 1.3in]{geometry}
\numberwithin{equation}{section}

\newcommand{\C}{\mathbb{C}}
\newcommand{\N}{\mathbb{N}}
\newcommand{\Q}{\mathbb{Q}}

\newcommand{\Z}{\mathbb{Z}}

\newcommand{\F}{\mathbb{F}}

\newcommand{\m}{\mathfrak{m}}
\newcommand{\n}{\mathfrak{n}}
\newcommand{\p}{\mathfrak{p}}

\newcommand{\mfn}{\mathfrak{n}}

\newcommand{\mfp}{\mathfrak{p}}

\newcommand{\SL}{\mathrm{SL}}
\newcommand{\GL}{\mathrm{GL}}

\newcommand{\ra}{\rightarrow}
\newcommand{\ras}{\twoheadrightarrow}
\newcommand{\mrm}[1]{\mathrm{#1}}

\def\1{1\!\!1}

\newcommand{\psmat}[4]{\bigl( \begin{smallmatrix} #1 & #2 \\ #3 & #4 \end{smallmatrix} \bigr)}

\title[Drinfeld modular forms for $\Gamma_0(T)$]{The  structure of Drinfeld modular forms of level $\Gamma_0(T)$ and applications}
\author[T. Dalal]{Tarun Dalal}
\email{ma17resch11005@iith.ac.in}
\address{
Department of Mathematics \\
Indian Institute of Technology Hyderabad\\
Kandi, Sangareddy - 502285\\
INDIA. 
}
\author[N. Kumar]{Narasimha Kumar}
\email{narasimha@math.iith.ac.in}
\address{
Department of Mathematics \\
Indian Institute of Technology Hyderabad\\
Kandi, Sangareddy - 502285\\
INDIA. 
}

\keywords{Drinfeld modular forms, Graded algebra structure, Isobaric polynomials, Weight filtration, Mod-$\p$ congruences}
\subjclass[2010]{Primary 11G09, 11F52; Secondary 11F33, 13C05}
\date{\today}
\begin{document}
\begin{abstract}
In this article, we describe the structure of the $R$-algebra of Drinfeld modular forms $M(\Gamma_0(T))_R$
(resp., $M^0(\Gamma_0(T))_R$) of level $\Gamma_0(T)$ and  the structure of mod-$\p$ reduction of $M_{\mfp}^0(\Gamma_0(T))$ for $\p \neq (T)$. As a result, we are able to study the properties of the weight filtration 
for $M_{k,l}(\Gamma_0(T))$. Finally, we prove a result on mod-$\p$ congruences for Drinfeld modular forms of level $\Gamma_0(\p T)$ for $\p \neq (T)$.
\end{abstract}

\maketitle

\section{Introduction}
It is well-known that the $\C$-algebra of modular forms for $\SL_2(\Z)$ is generated by $E_4$ and $E_6$,
where $E_4,E_6$ denote the Eisenstein series of weight $4$, $6$, respectively. As a consequence, Serre, Swinnerton-Dyer studied the properties of the weight filtration for $\SL_2(\Z)$ (cf.~\cite{Ser71}, ~\cite{Swi73}). 
Similarly, the $\C$-algebra of modular forms for $\Gamma_0(2)$ is generated by $E_2^*(q):= E_2(q)- 2E_2(q^2)$, $E_4$. There are some important works, in the literature, about the structure of the $\C$-algebra of classical modular forms for higher levels (cf.~\cite{Rus14},~\cite{VZ} for more details). 

It is of natural interest to understand the structure of the $C$-algebra of Drinfeld modular forms of level $\Gamma_0(\n)$. Such a study has been initiated by Gekeler in~\cite{Gek88}, where he proved that 
the $C$-algebra $M(\GL_2(A))$, Drinfeld modular forms of level $\GL_2(A)$, is isomorphic to the polynomial ring in two variables $C[X,Y]$. Hence, 
to each $f \in M(\GL_2(A))$, one can  attach a unique polynomial in $C[X,Y]$.
Using these polynomials and the mod-$\p$ reduction of the ring $M_\p(\GL_2(A))$,  Drinfeld modular forms 
of level $\GL_2(A)$ with $\p$-integral $u$-expansion at $\infty$, Vincent (\cite{Vin10}) studied the properties of the weight filtration 
for $M_{k,l}(\GL_2(A))$. 
To the best of  authors knowledge, the analogues of these results for higher levels is unknown.

In this article, we  show that $R$-algebra $M(\Gamma_0(T))_R$,  Drinfeld modular forms
of level $\Gamma_0(T)$ with coefficients in $R$, is generated by $\Delta_W, \Delta_T, E_T$
and it is the quotient of a polynomial ring in $3$-variables $R[U,V,Z]$.  Hence, every $f \in M(\Gamma_0(T))_R$ can  be  expressed as a polynomial in $R[U,V,Z]$, but it is not unique. 

The novelty in our work is to show that the properties of the weight filtration for $M_{k,l}(\Gamma_0(T))$ depends only on the structure of $M^0(\Gamma_0(T))_R$,  Drinfeld modular forms of type $0$, level $\Gamma_0(T)$ with coefficients in $R$.
So, we describe the structure of $M^0(\Gamma_0(T))_R$ and show that it is a polynomial ring in two variables $R[U,V]$. Hence, for  $f \in M_{k,0}(\Gamma_0(T))_R$, one can attach a unique polynomial $\varphi_f(U,V) \in R[U,V]$ such that $f=\varphi_f(\Delta_W, \Delta_T)$.
Using these polynomials and the mod-$\p$ reduction of the ring $M_\p^0(\Gamma_0(T))$, Drinfeld modular forms of type $0$, level $\Gamma_0(T)$ with $\p$-integral $u$-expansion at $\infty$, we study the properties of the weight filtration for $M_{k,l}(\Gamma_0(T))$.

Finally, as an application, we shall prove a result on mod-$\p$ congruences for Drinfeld modular forms of level $\Gamma_0(\p T)$ for $\p \neq (T)$.




%

\subsection{Notations:}
\label{notations}
Throughout the article, we fix to use the following notations.

Let $p$ be an odd prime number and $q=p^r$ for some $r \in \N$.
Let $\F_q$ denote the finite field of order $q$. We set $A :=\F_q[T]$ 
and $K :=\F_q(T)$.
Let $K_\infty=\F_q((\frac{1}{T}))$ be the completion of $K$ 
with respect to the infinite place $\infty$ (corresponding to $\frac{1}{T}$-adic valuation) 
and the completion of an algebraic closure of $K_{\infty}$ is denoted by $C$. 

\begin{itemize}
\item Let $k\in \N\cup \{0\}$ and $l\in \Z/(q-1)\Z$ such that $k\equiv 2l \pmod {q-1}$.
      Let $0 \leq l \leq q-2$ be a lift of $l \in \Z/(q-1)\Z$. By abuse of notation,
      we continue to write $l$ for the integer as well as its class. Then, we define 
      $r_{k, l} := \frac{k-2l}{q-1}$.
\item Let $R$ be a ring such that $A \subseteq R \subseteq C$.
\item $\mfp$ denotes a prime ideal of $A$ generated by a monic irreducible polynomial $\pi$ of degree $d$  such that 
$\p \ne (T)$.

\end{itemize}


\section{Preliminaries}
\label{Basic theory of Drinfeld modular forms section}
In this section, we recall some basic theory of Drinfeld modular forms (cf. \cite{Gos80}, \cite{Gos80a}, \cite{Gek88}, \cite{GR96} for more details).

Let $L=\tilde{\pi}A \subseteq C$ be the $A$-lattice of rank $1$,
corresponding to the rank $1$ Drinfeld module given by $\rho_T=TX+X^q$, which is also known as the Carlitz module, where $\tilde{\pi}\in K_\infty(\sqrt[q-1]{-T})$ is defined up to a $(q-1)$-th root of unity.
  Any $x\in K_\infty^\times$ has the unique expression 
$x= \zeta_x\big(\frac{1}{T} \big)^{v_\infty(x)}u_x,$
where $\zeta_x\in \F_q^\times$, and $v_\infty(u_x-1)\geq 0$ ($v_\infty$ is the valuation at $\infty$).
The Drinfeld upper half-plane $\Omega= C-K_\infty$ has a rigid analytic structure.
The group $\GL_2(K_\infty)$ acts on $\Omega$ via fractional linear transformations.
For $\gamma = \psmat{a}{b}{c}{d}\in \GL_2(K_{\infty})$ and $f:\Omega \ra C$,  we define 
$f|_{k,l} \gamma := \zeta_{\det\gamma}^l\big(\frac{\det \gamma}{\zeta_{\det(\gamma)}} \big)^{k/2}(cz+d)^{-k}f(\gamma z).$ 
For an ideal $\mfn \subseteq A$, we define the congruence subgroup $\Gamma_0(\mfn)$ of $\GL_2(A)$ 
by $\{\psmat{a}{b}{c}{d}\in \mathrm{GL}_2(A) : c\in \mfn \}$.
\begin{dfn}
\label{Definition of DMF}
Let $k\in \N\cup \{0\},l\in \Z/(q-1)\Z$.
A rigid holomorphic function $f:\Omega \ra C$ is said to be a Drinfeld modular form of 
weight $k$, type $l$ for $\Gamma_0(\mfn)$ if 
\begin{enumerate}
\item $f|_{k,l}\gamma= f$ , $\forall \gamma\in \Gamma_0(\mfn)$, 

\item $f$ is holomorphic at the cusps of $\Gamma_0(\mfn)$.
\end{enumerate}
The space of Drinfeld modular forms of weight $k$, type $l$ for $\Gamma_0(\mfn)$ is denoted by $M_{k,l}(\Gamma_0(\mfn)).$
Furthermore, if $f$ vanishes (resp., vanishes at least twice) at the cusps of $\Gamma_0(\mfn)$, then we say $f$ is a Drinfeld cusp (resp., doubly cuspidal) form of weight $k$, type $l$ for $\Gamma_0(\mfn)$. These spaces are denoted by $M^1_{k,l}(\Gamma_0(\mfn))$ (resp., $M_{k,l}^2(\Gamma_0(\n))$).
\end{dfn}

Note that, if $k\not \equiv 2l \pmod {q-1}$ then $M_{k,l}(\Gamma_0(\mfn))=\{0\}$. So, without loss of generality, we   assume that $k\equiv 2l \pmod {q-1}$.

Let $u(z) :=  \frac{1}{e_L(\tilde{\pi}z)}$, where $e_L(z):= z{\prod_{\substack{0 \ne \lambda \in L }}}(1-\frac{z}{\lambda}) $
is the exponential function attached to the lattice $L$, which is associated with the Carlitz module. 
Then, each Drinfeld modular form  $f\in M_{k,l}(\Gamma_0(\mfn))$ has a unique $u$-expansion at $\infty$
given by $f=\sum_{i=0}^\infty a_f(i)u^i$.
Since $\psmat{\zeta}{0}{0}{1} \in \Gamma_0(\n)$ for $\zeta \in \F_q^\times$, condition $(1)$ of Definition~\ref{Definition of DMF} 
implies $a_f(i)=0$ if $i\not\equiv l \pmod {q-1}$.
Hence, the $u$-expansion of $f$ at $\infty$ can be written as
$\sum_{0 \leq \ i \equiv l \mod (q-1)} a_f(i)u^{i}.$ 
Note that any Drinfeld modular form of type $> 0$ is automatically a cusp form. 

We define $M_{k,l}(\Gamma_0(\n))_R$ to be the space of Drinfeld modular forms $f$ of weight $k$, type $l$ for $\Gamma_0(\n)$ such that the $u$-expansion of $f$ at $\infty$ belongs to $R[[u]]$ (or we simply say  ``with  coefficients in $R$'').
For any $0 \neq f\in M_{k,l}(\Gamma_0(\n))_R$, we define $f^0 :=1$. 

We end this discussion by introducing the notion of $\p$-adic valuation of $f$. This makes it easier to define a congruence between two Drinfeld modular forms.



\begin{dfn}
Let  $f=\sum_{n\geq 0} a_f(n)u^n$ be a formal $u$-series in $K[[u]]$. 
We define $v_\mfp(f) := \inf_n v_\mfp (a_f(n)),$ 
where $v_\mfp(a_f(n))$ is the $\mfp$-adic valuation of $a_f(n)$.  
We say $f$ has a $\p$-integral $u$-expansion if $v_\p(f)\geq 0$.
\end{dfn}

\begin{dfn}[Congruence]
\label{definition for congruence of modular forms}
Let $f= \sum_{n\geq 0}a_f(n)u^n$ and $g= \sum_{n \geq 0}a_g(n)u^n$ be two  $u$-formal  series in $K[[u]]$. 
We say that $f\equiv g \pmod \p$ if $v_\p(f-g) \geq 1$.
\end{dfn}

\subsection{Examples}
We now give some examples of Drinfeld modular forms. 
\begin{exple}[\cite{Gos80}, \cite{Gek88}]
\label{Eisenstein Series}
Let $d\in \N$. For $z\in \Omega$, the function
\begin{equation*}
g_d(z) := (-1)^{d+1}\tilde{\pi}^{1-q^d}L_d \sum_{\substack{a,b\in \F_q[T] \\ (a,b)\ne (0,0)}} \frac{1}{(az+b)^{q^d-1}}
\end{equation*}
is a Drinfeld modular form of weight $q^d-1$, type $0$ for $\mathrm{GL}_2(A)$,
where $\tilde{\pi}$ is the Carlitz period and $L_d:=(T^q-T)\cdots(T^{q^d}-T)$ is the least common multiple of all monic polynomials of degree $d$.
We refer to $g_d$ as an Eisenstein series.
\end{exple}

\begin{exple}[\cite{Gos80a}, \cite{Gek88}]
 \label{Delta-function}
For $z\in \Omega$, the function
\begin{equation*}
\Delta(z) := (T-T^{q^2})\tilde{\pi}^{1-q^2}E_{q^2-1} + (T^q-T)^q\tilde{\pi}^{1-q^2}(E_{q-1})^{q+1},
\end{equation*}
is a Drinfeld cusp form of weight $q^2-1$,  type $0$ for $\mathrm{GL}_2(A)$,
where $E_k(z)= \sum_{\substack{(0,0)\ne (a,b)\in A^2}}\frac{1}{(az+b)^k}.$
The $u$-expansion of $\Delta$ at $\infty$ is given by $-u^{q-1}+ \cdots$.  
\end{exple}
\begin{exple}[\cite{Gek88}]
\label{Poincare Series}
For $z \in \Omega$, the function
\begin{equation*}
 h(z) := \sum_{\gamma = \psmat{a}{b}{c}{d}\in H\char`\\ \GL_2(A)} \frac{\det \gamma . u(\gamma z)}{(cz+d)^{q+1}},
\end{equation*}
is a Drinfeld cusp form of weight $q+1$, type $1$ for $\mathrm{GL}_2(A)$,
where $H=\big\{\psmat{*}{*}{0}{1}\in \GL_2(A)\big\}$.
The $u$-expansion of $h$ at  $\infty$ is given by $ -u -u^{(q-1)^2+1}+\cdots$.

\end{exple}

We end this section by introducing an important function $E$, which is not modular.
 In~\cite{Gek88}, Gekeler defined the function $E(z):= \frac{1}{\tilde{\pi}} \sum_{\substack{a\in \F_q[T] \\ a \ \mathrm{monic}}} ( \sum_{b\in \F_q[T]} \frac{a}{az+b} )$
 which is analogous to the Eisenstein series of weight $2$ over $\Q$. Using $E$, one can construct the Drinfeld modular form
 \begin{equation}
 \label{E_T}
 E_T(z) := E(z)- TE(Tz) \in M_{2,1}(\Gamma_0(T)).
 \end{equation}
The $u$-expansion of $E_T$ at $\infty$ is given by $u-Tu^q+\cdots\in A[[u]]$
(cf.~\cite[Proposition 3.3]{DK}).
Since $\mrm{dim}_C M_{2,1}(\Gamma_0(T)) = 1$, we have $M_{2,1}(\Gamma_0(T)) = \langle E_T \rangle$.

\section{Statements of the main results}
In this section, we  state the main results of this article (cf. Theorem~\ref{MT_1}, Theorem~\ref{MT_2}, 
Theorem~\ref{MT_3}).

\subsection{On the structure of $M(\Gamma_0(T))_R$:}
Let $M(\Gamma_0(\n))_R := \oplus_{k,l }  M_{k,l}(\Gamma_0(\n))_R$ be the graded $R$-algebra of Drinfeld modular forms 
of level $\Gamma_0(\n)$ with coefficients in $R$. In~\cite[Theorem 5.13]{Gek88}, Gekeler described the $C$-algebra structure of $M(\GL_2(A))$. In fact, he proved 
$M(\GL_2(A)) = C[g_1,h],$ where $g_1\in M_{q-1,0}(\GL_2(A))$, $h\in M_{q+1,1}(\GL_2(A))$ are as in Example~\ref{Eisenstein Series}, Example~\ref{Poincare Series}, respectively. In particular,  any $f\in M_{k,l}(\GL_2(A))$ can be expressed as a unique isobaric polynomial in $g_1$ and $h$. To the best of authors knowledge, the  structure of the algebra of Drinfeld modular forms for general $\Gamma_0(\mfn)$ is unknown.
In \S~\ref{structure_general}, we describe the structure of  $M(\Gamma_0(T))_R$. 
More precisely, we prove
\begin{thm}[Theorem \ref{MT_1} in the text]
\label{MT_1_Intro}

The  $R$-algebra $M(\Gamma_0(T))_R$ is generated by $\Delta_W, \Delta_T$ and $E_T$,
where $\Delta_W, \Delta_T$ and $E_T$ are as in~\eqref{delta_T and delta_W} and~\eqref{E_T}, respectively.
Moreover,
the surjective map $\vartheta(R) : R[U,V,Z] \ras  M(\Gamma_0(T))_R$ defined by  $U\rightarrow \Delta_W, V\rightarrow \Delta_T, Z\rightarrow E_T$ induces an isomorphism 
\begin{equation*}
R[U,V,Z]/(UV-Z^{q-1}) \cong M(\Gamma_0(T))_R.
\end{equation*}
\end{thm}
An immediate consequence of Theorem~\ref{MT_1_Intro} and the well-known
results from commutative algebra (cf.~\cite{Mat80}) would imply the following
corollary.
 \begin{cor}
\label{MT_1_Cor}
For any ring $R$ with $A\subseteq R \subseteq C$ we have $M(\Gamma_0(T))_R \cong R \otimes_A M(\Gamma_0(T))_A$. 
In fact, the Krull dimension of $M(\Gamma_0(T))_C$ is $2$.
\end{cor}

\subsection{On the structure of $M^0(\Gamma_0(T))_R$:}
Let $M^0(\Gamma_0(T))_R := \oplus_{k=0}^\infty M_{k,0}(\Gamma_0(T))_R $ denote the graded $R$-algebra of Drinfeld modular forms $f$ of type $0$, level $\Gamma_0(T)$ with coefficients in $R$.
In~\cite[Proposition 4.6.2]{Cor97}, Cornelissen described the graded $C$-algebra structure of $M^0(\Gamma_0(T))$
and showed that $M^0(\Gamma_0(T))= C[g_1(z), g_1(Tz)].$
However, these  generators do not describe the $A$-structure, i.e., $M^0(\Gamma_0(T))_A \supsetneq A[g_1(z), g_1(Tz)]$. In~\S \ref{Structure of M0(Gamma0(T)) subsection}, we specify a set of generators for the graded $A$-algebra  $M^0(\Gamma_0(T))_A$. 
More precisely, we prove:

\begin{thm}[Theorem \ref{MT_2} in the text]
\label{MT_2_Intro}
The map 
$$\vartheta_0 (R): R[U,V]\longrightarrow M^0(\Gamma_0(T))_R$$ 
defined by $(U,V) \rightarrow (\Delta_W, \Delta_T)$
is an isomorphism. 
In particular, the $R$-algebra $M^0(\Gamma_0(T))_R$ is generated by $\Delta_W, \Delta_T$,
i. e., $M^0(\Gamma_0(T))_R = R[\Delta_W, \Delta_T]$.
\end{thm}
An immediate consequence of Theorem~\ref{MT_2_Intro} and the well-known
results from commutative algebra (cf.~\cite{Mat80}) would imply that
 \begin{cor}
\label{MT_2_Cor}
For any ring $R$ with $A\subseteq R \subseteq C$, we have $M^0(\Gamma_0(T))_R \cong R \otimes_A M^0(\Gamma_0(T))_A$. Moreover, if $R$ is a Noetherian local ring, then the Krull dimension of $M^0(\Gamma_0(T))_R$ is $\dim(R)+2$.
\end{cor}

%

%

\subsection{On the mod-$\p$ reduction:}
Let $M_\p(\GL_2(A))$ denote the space of Drinfeld modular forms for $\mathrm{GL}_2(A)$ 
with $\p$-integral $u$-expansion at $\infty$. Define
$$\overline{M_{\p}}(\GL_2(A)) := \{\overline{f} \in \F_\p[[u]] \mid \exists f \in M_\p(\GL_2(A)) \ \mathrm{such\ that} \ f\equiv \overline{f} \pmod\p \},$$
where $\F_\p:= A_\p/\p A_\p$ and $A_\p$ is the localization of $A$ at $\p$. The map  
$\epsilon: \F_\p[X,Y] \ra \overline{M_\p}(\GL_2(A))$
defined by  $X \ra \overline{g}_1$, $Y \ra \overline{h}$
is surjective and 
$\mathrm{ker}\ \epsilon= (\overline{A}_d(X,Y)-1)$, where $A_d(X,Y)$ is the unique 
isobaric polynomial attached to $g_d$ and $``-"$ denotes the reduction mod $\p$
(cf.~\cite[Corollary 12.4]{Gek88}). In \S\ref{Graded structure section}, we extend this  result to Drinfeld modular forms of type $0$ for $\Gamma_0(T)$.

Let $M_\p^0(\Gamma_0(T))$ denote the space of Drinfeld modular forms of type $0$ for $\Gamma_0(T)$
 with $\p$-integral $u$-expansion at $\infty$. We define 
$$\overline{M_{\p}^0}(\Gamma_0(T)) := \{\overline{f} \in \F_\p[[u]] \mid \exists f \in M_\p^0(\Gamma_0(T)) \ 
\mathrm{such\ that} \ f\equiv \overline{f} \pmod\p \}.$$
  
\begin{thm}[Theorem~\ref{MT_3} in text]
\label{MT_3_Intro}
The surjective map $\iota : \F_\p[U,V]  \ras \overline{M_{\p}^0}(\Gamma_0(T))$
defined by $U \ra \overline{\Delta}_W$ and $V \ra  \overline{\Delta}_T$ induces an isomorphism 
\begin{equation*}
\F_\p[U,V]/(\overline{\phi}_d(U,V)-1) \cong \overline{M_{\p}^0}(\Gamma_0(T)),
\end{equation*}
where $\phi_d(U,V)$ is the unique isobaric polynomial attached to $g_d$, i.e., $g_d= \phi_d(\Delta_W, \Delta_T)$.
\end{thm}

\subsection{Weight filtration and its properties:}
By Theorem~\ref{MT_1_Intro}, for $f \in M_{k,l}(\Gamma_0(T))_R$, 
there exists a polynomial $\eta(U,V,Z)\in R[U,V,Z]$ such that $f=\eta(\Delta_W,\Delta_T,E_T)$.
Unfortunately, this polynomial is not unique, since $\mrm{ker}(\vartheta(R)) \neq 0$. 
By Theorem~\ref{MT_2_Intro}, 
for $f\in M_{k,0}(\Gamma_0(T))_R$, there exists a unique polynomial $\varphi_f(U,V) \in R[U,V]$ such that $f = \varphi_f(\Delta_W, \Delta_T)$. 
Using $\varphi_f(U,V)$  and Theorem~\ref{MT_3_Intro},
we study the properties of the weight filtration for $M_{k,l}(\Gamma_0(T))_{A_\p}$ (cf. Theorem~\ref{Lower_Filtration}). This can be thought of as a generalization of the work in~\cite{Vin10} 
from $\GL_2(A)$ to $\Gamma_0(T)$.

\subsection{Application:}
In the last section, as an application of the results of this article, 
we shall prove a result on mod-$\p$ congruences for Drinfeld modular forms of level $\Gamma_0(\p T)$
for $\p \neq (T)$, which is a variant of~\cite[Theorem 1.5]{DK}. 

Since the methods of this article work for any degree one prime,
it is of interest to study similar results for the remaining genus zero case i.e., for $\deg(\n)=2$, and also for  higher genus.

\section{Structure of Drinfeld modular forms for $\Gamma_0(T)$}
\label{Graded structure section}
In this section, first
we introduce two important Drinfeld modular forms and study their properties.
Recall that $0$ and $\infty$ are the only cusps of the Drinfeld modular curve $X_0(T):= \overline{\Gamma_0(T)\char`\\ \Omega}$ and the operator $W_T:=\psmat{0}{-1}{T}{0}$ permutes the cusps.
Consider the functions
\begin{equation}
\label{delta_T and delta_W}
 \Delta_T(z) := \frac{g_1(Tz)-g_1(z)}{T^q-T}\ \mrm{and} \ 
\Delta_W(z) := \frac{T^qg_1(Tz)-Tg_1(z)}{T^q-T}.
 \end{equation}
An easy check implies that $\Delta_T$, $\Delta_W\in M_{q-1,0}(\Gamma_0(T))$.
By~\cite[Page 3]{Cho09}, the $u$-expansion of $\Delta_T$ at $\infty$ is given by
$$\Delta_T=u^{q-1}-u^{q(q-1)}+\cdots\in A[[u]].$$
Since 
$g_1(z)=1-(T^q-T)u^{q-1}-(T^q-T)u^{(q^2-q+1)(q-1)}+ \cdots$ and 
$u(Tz)=u^q+\cdots$,
the $u$-expansion of $\Delta_W$ at $\infty$ is given by
\begin{equation*}
\Delta_W = 1+ Tu^{q-1} - T^qu^{q(q-1)} + \cdots \in A[[u]].
\end{equation*}
Next, we study  some important properties of $\Delta_T, \Delta_W$, and $E_T$.
Before, we state them as a proposition, first let us recall the dimension formula for $\Gamma_0(T)$.
\begin{prop}
\label{Dimension Gamma0(T)}
The $C$-dimension of $M_{k,l}(\Gamma_0(T))$ is $1+r_{k,l}$, where $r_{k,l} := \frac{k-2l}{q-1}$.
\end{prop}
\begin{proof}
If $l \equiv 0 \pmod {q-1}$, then the proposition follows from~\cite[Proposition 6.4]{Gek86}.
If $l \not \equiv 0 \pmod {q-1}$, then the proposition can be deduced from the discussion in~\cite[\S 4.3]{BV19},
since any Drinfeld modular form of non-trivial type is cuspidal. 
Now, we follow the notations as in~\cite{BV19}.

 Let $\mathcal{B}_k^1(\Gamma_1(T)):=\{c_j:0\leq j \leq k-2\}$ be a basis of $M_k^1(\Gamma_1(T))$ as in~\cite[\S 4.2]{BV19} and $C_j:=\{c_i: i\equiv j \pmod {q-1}\}.$ From the discussion in~\cite[\S 4.3]{BV19}, we get $M_{k,l}(\Gamma_0(T))=M_{k,l}^1(\Gamma_0(T))=\mathrm{span}\{c_i: c_i\in C_{l-1}\}.$ Hence 
$\dim M_{k,l}(\Gamma_0(T))=|C_{l-1}|=n=1+r_{k,l}.$
One can also use Riemann-Roch Theorem to give an alternate proof.
\end{proof}
We now define a  Victor Miller basis for $M_{k,l}(\Gamma_0(\mfn))$.
For any $f\in M_{k,l}(\Gamma_0(\n))$, suppose the $u$-expansion of $f$ at $\infty$ is given by $\sum_{i\geq 0} a_f({i(q-1)+l})u^{i(q-1)+l}$. Then, we define $b_f(i) := a_f({i(q-1)+l})$.
\begin{dfn}[Victor Miller basis]
We say that a basis ${f_1, f_2, \ldots, f_r}$
of $M_{k,l}(\Gamma_0(\n))$, where $r = \mrm{dim}_C M_{k,l}(\Gamma_0(\n))$, is a Victor Miller basis
for $M_{k,l}(\Gamma_0(\mfn))$ if $b_{f_j}(i) = \delta_{ij}$ for all $1 \leq i, j \leq r$.
\end{dfn}
      
We now state  some important properties of $\Delta_T, \Delta_W$, and $E_T$.
\begin{prop}
\label{IMP_Prop}
Let $\Delta_W, \Delta_T,  E_T$ be as in~\eqref{delta_T and delta_W} and~\eqref{E_T}. Then, we have  
\begin{enumerate}
\item \label{P1} $g_1=\Delta_W-T^q\Delta_T$.
\item \label{P2} 
       The modular form $\Delta_T$(resp., $\Delta_W$) vanishes $q-1$ times at  $\infty$ (resp., at  $0$) and non-zero on $\Omega \cup \{0\}$ (resp.,  on $\Omega \cup \{\infty\}$). Hence,  the modular forms $\Delta_T, \Delta_W$ are algebraically independent.
 \item \label{P3}
        The set $S:= \{\Delta_W^{r_{k,l}}E_T^l, \Delta_W^{r_{k,l}-1}\Delta_TE_T^l, \ldots, \Delta_W\Delta_T^{r_{k,l}-1} E_T^l, \Delta_T^{r_{k,l}}E_T^l \}$ is a basis for  $M_{k,l}(\Gamma_0(T))$. In particular, the space  $M_{k,l}(\Gamma_0(T))$ has an integral basis, i.e., a basis consisting of modular forms with coefficient in $A$.
\item \label{P8} There exists a Victor Miller basis for the $C$-vector space $M_{k,l}(\Gamma_0(T))$.
 
 \item  \label{P4}
        For  $k \in \N \cup \{0\}$, the map $\eta: M_{k-2l,0}(\Gamma_0(T)) \ra M_{k,l}(\Gamma_0(T))$ 
         defined by $f \ra fE_T^l $
         is an isomorphism. 
 \item \label{P5}
        We have $E_T^{q-1}= \Delta_W\Delta_T$.  In particular, the function $E_T\in M_{2,1}(\Gamma_0(T))$ vanishes exactly once at the cusps $0,\infty$ and non-vanishing  
        elsewhere.
 \item \label{P6}
        $\partial\Delta_W= -\Delta_WE_T$, $\partial \Delta_T=0$, $\partial(E_T)= E_T^2$,  
      where  $\partial f$ as in Definition~\ref{Definition_2nd Derivative}.
\item \label{P7} $h(z)= \partial(g_1(z))= \partial(g_1(Tz))= -\Delta_W(z)E_T(z)$, $\Delta= -\Delta_W^q\Delta_T$.
\end{enumerate}
\end{prop}
%
In literature, there is a folklore conjecture that there exists an integral  basis for $M_{k,l}(\Gamma_0(\n))$.  
The known results for $M_{k,l}(\GL_2(A))$, $M_{k,0}(\Gamma_0(T))$ are due to Gekeler (\cite{Gek88}), Cornelissen (\cite[Proposition 4.6.2]{Cor97}) respectively. For general level $\n$, the best known result is  
for $M_{2,1}^2(\Gamma_0(\n))$ due to Armana (\cite[Section 4.2]{Arm08}). Note that, Proposition~\ref{IMP_Prop}(\ref{P3})
implies  that the space  $M_{k,l}(\Gamma_0(T))$  has an integral basis. 

\begin{proof}[Proof of Proposition~\ref{IMP_Prop}]
\begin{enumerate}
\item This follows from the definitions of $\Delta_W$ and $\Delta_T$.
\item By $u$-expansion of $\Delta_T$ at $\infty$, we see that $\Delta_T$ vanishes exactly $q-1$ times at  $\infty$.
      By~\cite[Proposition 3.3]{Cho09}, the function $\Delta_T$ is non-zero on $\Omega \cup \{0\}$. 
      Since $\Delta_W= -T^\frac{q+1}{2}\Delta_T\mid_{q-1,0} W_T$ and $W_T$ permutes the cusps, 
      we see that $\Delta_W$ vanishes $q-1$ times at $0$ and non-zero on $\Omega \cup \{\infty\}$. 
      
      We now show that the functions   $\Delta_T, \Delta_W$ are algebraically independent.
      Suppose there is an algebraic relation between $\Delta_T$ and $\Delta_W$. The graded structure of $M(\Gamma_0(T))$ 
      implies that there exists an algebraic relation of minimal weight $k$ and of  the form
      \begin{equation}
      \label{minimal weight k}
      c_0\Delta_W^{\frac{k}{q-1}}+\sum_{i=1}^{n-1}c_i \Delta_W^{x_i}\Delta_T^{y_i} 
           + c_n\Delta_T^\frac{k}{q-1}=0,        
      \end{equation}
      with $(x_i+y_i)(q-1)=k$, $x_iy_i \neq  0$  for all $1 \leq i \leq n-1$, and $c_0c_n \neq 0$.
      Evaluating both sides of~\eqref{minimal weight k} at $\infty$, we get $c_0=0$. This contradicts the minimality of $k$.

\item By Part (\ref{P2}), the set $S$ is linearly independent. Since $\# S = \mrm{dim}\ M_{k,l}(\Gamma_0(T))$,
      $S$ is a basis for the $C$-vector space $M_{k,l}(\Gamma_0(T))$. 
      Since the elements of $S$ have coefficients in $A$, $S$ is an integral basis for $M_{k,l}(\Gamma_0(T))$.

\item Let $S$ be as in Part (\ref{P3}). For $0 \leq j \leq r_{k,l}$, the $u$-expansion of elements of $S$ 
      is given by
      $$ f_j:= \Delta_W^{r_{k,l}-j}\Delta_T^j E_T^l = u^{j(q-1)+l}+\cdots \in A[[u]].$$
      Then, $B:= [b_{f_j}(i)]_{0 \leq j,i \leq r_{k,l}}$ is an upper triangular matrix  with diagonal entries $1$. Since the row reduced Echelon form of $B$ is an identity matrix, the existence of a Victor Miller basis for $M_{k,l}(\Gamma_0(T))$ follows.


\item  By Part (\ref{P3}),  the map  $\eta$ is surjective. By Proposition~\ref{Dimension Gamma0(T)},
        the domain, co-domain have same dimension, hence $\eta$ is an isomorphism.  
         

 \item  By~\cite[Theorem 6.3]{AW20}, it is enough to show that the coefficients of $u^0,u^{q-1}$ and $u^{2(q-1)}$ are equal in the $u$-expansions of $E_T^{q-1}$ and $\Delta_W\Delta_T$ at $\infty$. This is an easy check.
  The latter part follows from the vanishing (resp., non-vanishing) of $\Delta_T$, $\Delta_W$ on $\Omega\cup\{0,\infty\}$.
 
 
\item By~\cite[Theorem 6.3]{AW20}, it is enough to show that the coefficients of $u$ and $u^q$ are equal in the $u$-expansions of $\partial \Delta_T$ and $0$ (resp., $\partial\Delta_W$ and $\Delta_WE_T$ ) at $\infty$. The third equality can be verified by showing that the coefficients of $u^2$ are equal in the $u$-expansions of $\partial(E_T)$ and $E_T^2$. This is an easy check.


\item By~\cite[Theorem 6.3]{AW20}, it is enough to show that the coefficients of $u$ and $u^q$ are equal in the $u$-expansions of $h(z),\partial(g_1(z)), \partial(g_1(Tz))$ and $\Delta_W(z)E_T(z)$ at $\infty$. This is an easy check. The last relation follows from Part(\ref{P5}) and $h^{q-1}=-\Delta$.

   \end{enumerate}
\end{proof}

\subsection{On the structure of $M(\Gamma_0(T))_R$:} 
\label{structure_general}
We are now ready to prove one of the main results of this article, which is a generalization of~\cite[Theorem 5.13]{Gek88} from $\GL_2(A)$ to $\Gamma_0(T)$. 

%

\begin{thm}
\label{MT_1}
The  $R$-algebra $M(\Gamma_0(T))_R$ is generated by $\Delta_W, \Delta_T$ and $E_T$. Moreover,
the surjective map $\vartheta(R) : R[U,V,Z] \ras  M(\Gamma_0(T))_R$ defined by  $U\rightarrow \Delta_W, V\rightarrow \Delta_T, Z\rightarrow E_T$ induces an isomorphism 
\begin{equation}
R[U,V,Z]/(UV-Z^{q-1}) \cong M(\Gamma_0(T))_R.
\end{equation}
\end{thm}
To prove Theorem~\ref{MT_1}, first we need to introduce the notion of weight, type for a polynomial in $R[U,V,Z]$.
\begin{dfn}
\label{Definition of weight and type of a polynomial}
Let $k\in \N\cup \{0\}$ and $l\in \Z/(q-1)\Z$.
A polynomial $\varphi(U,V,Z)= \sum_{i=1}^n c_i U^{x_i}V^{y_i}Z^{l_i}\in R[U,V,Z]$ is said to be of 
weight $k$, type $l$  if 
\begin{enumerate}
\item $(x_i+y_i)(q-1)+2l_i=k$, $ \forall 1\leq i \leq n$ and
\item $l_1 \equiv l_2 \equiv \cdots \equiv l_n \equiv l \pmod {q-1}$.
\end{enumerate}
\end{dfn}
\begin{remark}
Definition~\ref{Definition of weight and type of a polynomial} is also valid for polynomials over $\F_\mfp$.
By definition, the variables $U, V$ are of weight $q-1$, type $0$ and $Z$ is of weight $2$, type $1$.
\end{remark}

Let $R[U,V,Z]_{k,l}$ denote the set of polynomials of weight $k$, type $l$. Then
\begin{equation}
\label{decomposition of C[U,V,Z]}
R[U,V,Z] = \bigoplus_{\substack{k\in \N\cup \{0\}\\l\in \Z/(q-1)\Z}} R[U,V,Z]_{k,l}.
\end{equation}

\begin{proof}[Proof of Theorem~\ref{MT_1}]
For simplicity, we denote $\vartheta(R)$ by $\vartheta$.
Let $f\in M_{k,l}(\Gamma_0(T))_R$.
 Since $B:=\{\Delta_W^{r_{k,l}-j}\Delta_T^j E_T^l| 0 \leq j \leq r_{k,l} \}$ is a basis for the $C$-vector space $M_{k,l}(\Gamma_0(T))$ (cf. Proposition~\ref{IMP_Prop}~\eqref{P3}),  $f$ can be written as
  \begin{equation}
  \label{expression of modular forms of any type}
  f= c_0\Delta_W^{r_{k,l}}E_T^l+ c_1\Delta_W^{r_{k,l}-1}\Delta_TE_T^l+ \cdots+c_{r_{k,l}} \Delta_T^{r_{k,l}} E_T^l  \ \mathrm{with} \ c_j\in C.
  \end{equation}
  For $0 \leq j \leq r_{k,l}$, the $u$-expansion of elements of $B$ is given by
  $$ \Delta_W^{r_{k,l}-j}\Delta_T^j E_T^l = u^{j(q-1)+l}+\cdots.$$
  Comparing the coefficients in $u$-expansions on both sides of~\eqref{expression of modular forms of any type} we get $c_j\in R$ for all $j\in \{0,\ldots, r_{k,l}\}$. This shows that the map $\vartheta$ is surjective. 
In particular, this implies that the  $R$-algebra $M(\Gamma_0(T))_R$ is generated by $\Delta_W, \Delta_T$ and $E_T$.


It remains to show that $ \ker\vartheta = (UV-Z^{q-1})$.
By Proposition~\ref{IMP_Prop}(\ref{P5}), we get $(UV-Z^{q-1}) \subseteq \ker\vartheta$.
To show the equality, by~\eqref{decomposition of C[U,V,Z]}, it suffices to show that for any $f\in R[U,V,Z]_{k,l}$ such that $\vartheta(f)=0$ then $f\in (UV-Z^{q-1})$. Recall that $0\leq l\leq q-2$ is a lift of $l\in \Z/(q-1)\Z$.

Let $f=\sum_{i=0}^n c_iU^{x_i}V^{y_i}Z^{l_i}\in R[U,V,Z]_{k,l}$ be a polynomial of weight $k$, type $l$
such that $\vartheta(f)=0$.
The condition $(2)$ in Definition~\ref{Definition of weight and type of a polynomial} implies that $f$ can be written as $f = \sum_{i=0}^n c_iU^{x_i}V^{y_i}Z^{l+m_i(q-1)}$ for some $m_i \in \Z$.
By Proposition~\ref{IMP_Prop}(\ref{P5}), we have
\begin{equation}
\label{Image of weight k type l polynomial to a modular form}
\vartheta(f)=E_T^l\cdot\sum_{i=0}^n c_i\Delta_W^{x_i+m_i}\Delta_T^{y_i+m_i}=0.
\end{equation}
Define  $S:= \{y_i+m_i \mid 0\leq i \leq n\}$, $T:= \{x_i+m_i \mid 0\leq i \leq n\}$. 
For $(a,b)\in S\times T$, we define 
\begin{equation}
\label{relation between the exponents for distinct sum}
 I_{(a,b)} :=\{i\ | \ y_i+m_i=a, \ x_i+m_i=b\}.
\end{equation}
The expression in~\eqref{Image of weight k type l polynomial to a modular form} is reduced to 
$E_T^l\cdot \sum_{(a,b) \in S\times T} c_{a,b} \Delta_W^b\Delta_T^a=0,$
where $c_{a,b}= \sum_{i\in I_{(a,b)}} c_i$.
Since $\Delta_W$, $\Delta_T$ are algebraically independent (cf. Proposition~\ref{IMP_Prop}~\eqref{P2})
we get 
\begin{equation}
\label{c_a,b =0 gives sum of c_j;a,b=0}
c_{a,b}= \sum_{i\in I_{(a,b)}} c_i=0 \quad  \mrm{for} \ (a,b)\in S \times  T.
\end{equation}
Let $f_{a,b} := \sum_{i\in I_{(a,b)}} c_i U^{x_i}V^{y_i}Z^{l+{m_i}(q-1)}$ for   
$(a,b)\in S \times  T$. Since $f= \sum_{(a,b)\in S \times  T} f_{a,b}$, it is enough to show that $f_{a,b} \in (UV-Z^{q-1})$.

Let $m_r = \inf_{i\in I_{(a,b)}} m_i$. 
By~\eqref{c_a,b =0 gives sum of c_j;a,b=0}, the expression for $f_{a,b}$ becomes
\begin{equation}
\label{Big_expression}
\sum_{\substack{i\in I_{(a,b)}\\i\ne r}} c_i U^{x_i}V^{y_i}Z^{l+{m_i(q-1)}} - (\sum_{\substack{{i\in I_{(a,b)}}\\ i\neq r}} c_i)
U^{x_r} V^{y_r}Z^{l+{m_r}(q-1)}. 
\end{equation}
For simplicity, we let $x(j;i) := x_j-x_i$,  $y(j;i) := y_j-y_i$, and $m(i;j) := m_i-m_j$.
Then, the expression in~\eqref{Big_expression} reduced to
$$Z^l \sum_{\substack{i\in I_{(a,b)}\\i\ne r}} c_iZ^{m_r(q-1)}U^{x_i}V^{y_i} 
\Big(Z^{(q-1)m(i;r)} - U^{x(r;i)} V^{y(r;i)}\Big)$$

By~\eqref{relation between the exponents for distinct sum}, we see that 
$x(r;i) =  y(r;i) = m(i;r)$ for $i\in I_{(a,b)}$.
Hence, the expression in the last bracket becomes 
$Z^{(q-1)m(i;r)} - (UV)^{m(i;r)}$, which is a multiple of $Z^{q-1}-UV$.
This implies $f_{a,b}\in (UV-Z^{q-1})$  for all $(a,b) \in S \times T$
and this proves the claim.
\end{proof}

 By Theorem~\ref{MT_1}, for  $f\in M_{k,l}(\Gamma_0(T))_R$, there exists a polynomial $\eta(U,V,Z)\in R[U,V,Z]$ such that $f=\eta(\Delta_W,\Delta_T,E_T)$. Unfortunately, the polynomial $\eta$ is not unique since $\ker \vartheta(R) \neq 0$. 
 
 In the next section, we show that for each $f\in M_{k,l}(\Gamma_0(T))_R$ there exists a unique polynomial $\varphi_f(U,V)\in R[U,V]$ such that $f = \varphi_{f}(\Delta_W, \Delta_T) E_T^l$.
To do this, we need to describe the structure of the $R$-algebra $M^0(\Gamma_0(T))_R$.

\subsection{On the structure of $M^0(\Gamma_0(T))_R$:}
\label{Structure of M0(Gamma0(T)) subsection}
As a consequence of Proposition~\ref{IMP_Prop}, we produce an explicit set of generators for 
the $R$-algebra $M^0(\Gamma_0(T))_R$. 
\begin{thm}
\label{MT_2}
The map 
$$\vartheta_0 (R): R[U,V]\longrightarrow M^0(\Gamma_0(T))_R$$ 
defined by $(U,V) \rightarrow (\Delta_W, \Delta_T)$
is an isomorphism. 
In particular, the $R$-algebra $M^0(\Gamma_0(T))_R$ is generated by $\Delta_W, \Delta_T$,
i. e., $M^0(\Gamma_0(T))_R = R[\Delta_W, \Delta_T]$.

\end{thm}

\begin{proof}
The surjectivity of $\vartheta_0 (R)$ is clear from~\eqref{expression of modular forms of any type} by choosing $l=0$. The injectivity of $\vartheta_0 (R)$ follows from Proposition~\ref{IMP_Prop}(\ref{P2}).
\end{proof}

We now introduce the notion of isobaric polynomial in $R[U,V]$.
\begin{dfn}
\label{Definition of weight and type zero of a polynomial}
A polynomial $\varphi(U,V)= \sum_{i=1}^nc_iU^{x_i}V^{y_i}\in R[U,V]$ is said to be an isobaric polynomial of weight $k\in \N\cup \{0\}$ if $(x_i+y_i)(q-1)=k$ for all $1\leq i \leq n$.  
\end{dfn}
\begin{remark}
Definition~\ref{Definition of weight and type zero of a polynomial} is also valid for polynomials over $\F_\mfp$.
By definition, the variables $U, V$ are isobaric polynomials of weight $q-1$.
\end{remark}




By Theorem~\ref{MT_2}, for  $f\in M_{k,0}(\Gamma_0(T))_R$,
there exists a unique polynomial $\varphi_f(U,V)=\sum_{i=1}^nc_iU^{x_i}V^{y_i}\in R[U,V]$ such that $f= \varphi_f(\Delta_W, \Delta_T)=\sum_{i=1}^nc_i\Delta_W^{x_i}\Delta_T^{y_i}$. Since the weight of $f$ is $k$, we  have $(x_i+y_i)(q-1)=k$ for all $1\leq i \leq n$. Hence the polynomial $\varphi_f(U,V)$ is an isobaric polynomial of weight $k$. 
These polynomials play an important role in rest of the article.

The novelty of our work is that we use Theorem~\ref{MT_2} to show that for each $f\in M_{k,l}(\Gamma_0(T))_R$ there exists a unique polynomial $\varphi_f(U,V) \in R[U,V]$
such that $f = \varphi_f(\Delta_W, \Delta_T) E_T^l$. This is  done as follows.

To each $f\in M_{k,l}(\Gamma_0(T))_R$, we define $c_T(f) :=\frac{f}{E_T^l}\in M_{k-2l,0}(\Gamma_0(T))$ (cf. Proposition~\ref{IMP_Prop}(\ref{P4})). 
Since the $u$-expansion of $E_T^l$ at $\infty$ is $u^l(1+\cdots)\in A[[u]]$, 
an easy verification shows that $c_T(f)\in M_{k-2l,0}(\Gamma_0(T))_R$.
By Theorem~\ref{MT_2}, there exists a unique isobaric polynomial $\varphi_{c_T(f)}(U,V)\in R[U,V]$ of weight $k-2l$ such that 
$c_T(f)=\varphi_{c_T(f)}(\Delta_W, \Delta_T).$ This implies that $f = \varphi_{c_T(f)}(\Delta_W, \Delta_T) E_T^l$.
For simplicity, we denote $\varphi_{c_T(f)}(U,V)$ by $\varphi_{f}(U,V)$. We summarize the  discussion above as a proposition.
 

\begin{prop}
\label{coefficient of the u-series and the polynomial for arbitrary type}
\label{Isobaric polynomial for arbitrary type}
For $f\in M_{k,l}(\Gamma_0(T))_R$, there exists a unique isobaric polynomial $\varphi_{f}(U,V)\in R[U,V]$
 such that $c_T(f)=\varphi_{f}(\Delta_W, \Delta_T)$. Hence, $f=\varphi_{f}(\Delta_W, \Delta_T)E_T^l$. 
\end{prop}

We now record a useful lemma.
\begin{lem}
\label{Congruence between forms of arbitrary type resuces to congruence between trivial type}
	For $i=1,2$, if $f_i\in M_{k_i,l_i}(\Gamma_0(T))$ such that at least one of $f_1,f_2$ is non-zero modulo $\mfp$.
	Then, $f_1 \equiv f_2 \pmod \p$ if and only if $l_1\equiv l_2\pmod {q-1}$ and $c_T(f_1)\equiv c_T(f_2) \pmod \p$.
\end{lem}
\begin{proof}
Let $0\leq l \leq q-2$ be an integer such that $l_1 \equiv l_2 \equiv l \pmod {q-1}$  and $c_T(f_1)\equiv c_T(f_2) \pmod \p$. Since $E_T^l$ has coefficients in $A$, the congruence $c_T(f_1)\equiv c_T(f_2) \pmod \p$ gives $c_T(f_1)E_T^l\equiv c_T(f_2)E_T^l \pmod \p$, i.e., $f_1\equiv f_2 \pmod \p$.

Conversely, we suppose  $f_1\equiv f_2 \pmod \p$. 
The $u$-expansions of $f_1$ and $f_2$ imply that $l_1\equiv l_2 \equiv l \pmod {q-1}$ for $0\leq l \leq q-2$.  Since $f_i=c_T(f_i)E_T^l$,
one gets $c_T(f_1)E_T^l\equiv c_T(f_2)E_T^l \pmod \p$.
The $u$-expansion of $E_T^l$ at $\infty$ is $u^l(1+\cdots)\in A[[u]]$, so we obtain $c_T(f)\equiv c_T(g) \pmod \p$. 
\end{proof}
By Lemma~\ref{Congruence between forms of arbitrary type resuces to congruence between trivial type}, 
the  congruence between Drinfeld modular forms of arbitrary type reduces to the congruence between Drinfeld modular forms of type $0$.
Thus, the study of the weight filtration heavily depends on the structure of mod-$\p$ reduction of $M_\mfp^0(\Gamma_0(T))$, which is the content of the following section. 

\subsection{On the structure of mod-$\p$ reduction of $M_{\mfp}^0(\Gamma_0(T))$:}
\label{structure_particular_reduction}
We now describe the structure of $M_{\mfp}^0(\Gamma_0(T))$ under mod $\p$-reduction with $\p \ne (T)$. 

\begin{dfn}
Let $M_\p^0(\Gamma_0(T))$ denote the ring of Drinfeld modular forms for $\Gamma_0(T)$ of  type $0$
with $\p$-integral $u$-expansion at $\infty$. We define
$$\overline{M_{\p}^0}(\Gamma_0(T)) := \{\overline{f} \in \F_\p[[u]] \mid \exists f \in M_\p^0(\Gamma_0(T)) \ \mrm{such\  that} \ f\equiv \overline{f} \pmod\p \}.$$
\end{dfn}

Let $\phi_d(U,V)$ be the unique isobaric polynomial of weight $q^d-1$ attached to $g_d\in M_{q^d-1,0}(\Gamma_0(T))$ such that $g_d= \phi_d(\Delta_W, \Delta_T)$. Let $\overline{\phi}_d(U,V) := {\phi}_d(U,V) \pmod \p$.

\begin{thm}
\label{MT_3}
The surjective map $\iota : \F_\p[U,V]  \ras \overline{M_{\p}^0}(\Gamma_0(T))$
defined by $U \ra \overline{\Delta}_W$ and $V \ra  \overline{\Delta}_T$ induces an isomorphism 
\begin{equation*}
\F_\p[U,V]/(\overline{\phi}_d(U,V)-1) \cong \overline{M_{\p}^0}(\Gamma_0(T)).
\end{equation*}
\end{thm}

To prove Theorem~\ref{MT_3}, we  study the behavior of isobaric polynomials 
under a certain  change of variables and we show that $\overline{\phi}_d(U,V)$ is square-free.

 Let $B_d(X,Y_1)$ (resp., $ A_d(X,Y)$) be the unique isobaric polynomial attached to $g_d\in M_{q^d-1,0}(\GL_2(A))$  such that  $g_d=  B_d(g_1, \Delta) $ (resp., $g_d= A_d(g_1, h) $).   
 The relation between $\Delta,h$  implies  $Y_1= -Y^{q-1}$.
 By Proposition~\ref{IMP_Prop}, we get  $\phi_d(U,V) = B_d(U-T^qV, -U^qV)$.
By~\cite[Proposition 6.9]{Gek88}, the relation  among $g_d$ gives
\begin{equation}
\label{relation between phi_d}
{\phi}_d(U,V) = {\phi}_{d-1}(U,V)(U- T^qV)^{q^{d-1}} + (T^{q^{d-1}}-T){\phi}_{d-2}(U,V)(U^qV)^{q^{d-2}},
\end{equation}
where ${\phi}_0(U,V)=1$, ${\phi}_1(U,V)=U-T^qV$. 
\begin{remark}
\label{U,V does not divide phi_d remark}
From~\eqref{relation between phi_d}, we see that 
$U \mid {\phi}_d(U,V)$ if and only if $U \mid {\phi}_{d-1}(U,V)$. 
Since ${\phi}_1(U,V)=U-T^qV$, we get $U \nmid {\phi}_d(U,V)$ for all $d$.
Similarly, we get $U \nmid \overline{\phi}_d(U,V)$ for all $d$.
These properties do hold for the variable $V$.
\end{remark}

Consider the Tate-Drinfeld module TD over $K((u))$ defined by
$$ \mathrm{TD}_T= T\tau^0 + g_1(u)\tau + \Delta(u) \tau^2. $$
Let $\mathrm{TD}_\pi=\sum_{0\leq i\leq 2d}l_i \tau^i$ and $F_i(X,Y_1)\in A[X,Y_1]$ be 
the unique isobaric polynomial attached to $l_i$ such that $F_i(g_1, \Delta)= l_i.$
Define $$f_i(U,V) := F_i(U-T^qV, -U^qV).$$ 
By~\cite[Corollary 12.3]{Gek88}, we have
$\overline{\phi}_d(U,V)= \overline{f}_d(U,V)$.
We now show that the polynomial $\overline{\phi}_d(U,V)$ is square-free. The proof of this is 
analogous to the proof of~\cite[Proposition 11.7]{Gek88}.
\begin{prop}\label{square free}
The polynomial $\overline{\phi}_d(U,V)$ is square-free.
\end{prop}
\begin{proof}
Let $H(U) := \overline{f}_{2d-1}(U,1)$.
By~\cite[Lemma 11.6]{Gek88}, we get $H^\prime(U)= (-U)^a$ for some $a \in \N$. 
The possible multiple factors of $\overline{f}_{2d-1}(U,V)$ are  $U$ or $V$.
Since $\overline{f}_d \mid \overline{f}_{2d-1}$,
the possible multiple factors 
of $\overline{f}_d(U,V)$ are $U$ or $V$. This is a contradiction, 
since $\overline{\phi}_d(U,V)$ is  neither divisible by $U$ nor by $V$
(cf. Remark~\ref{U,V does not divide phi_d remark}). 
\end{proof}
\begin{remark}
\label{why p is different from T} 
If $\p=(T)$, then $\overline{\phi}_1(U,V) =  U$. So, $U|\overline{\phi}_1(U,V)$.
To avoid this situation, we work with a prime ideal $\p$ such that  $\p\ne (T)$.
\end{remark}

We have now all the information that is  required to prove Theorem~\ref{MT_3}.

\begin{proof}[Proof of Theorem \ref{MT_3}]
By Theorem~\ref{MT_2}, $\iota$ is surjective.
Clearly, $\mathrm{ker}(\iota)$ is a non-maximal prime ideal and $(\overline{\phi}_d(U,V)-1) \subseteq \mrm{ker}(\iota)$.
Since the Krull dimension of $\F_\p[U,V]$ is $2$, it suffices to show that $\overline{\phi}_d(U,V)-1$ is irreducible. 

Suppose $\overline{\phi}_d(U,V)-1$ is reducible. Let $\overline{\phi}_d(U,V)-1 = RS$, where  $R= \sum_{i\leq m}R_i$ and $S= \sum_{j\leq n}S_j$ where $R_i, S_j$ denote isobaric polynomials of weight $i,j$ respectively. By comparing the highest, lowest weights, we get $\overline{\phi}_d(U,V)= R_mS_n$, $R_0S_0=-1$, respectively. Since $\overline{\phi}_d(U,V)$ is square-free, the polynomials $R_m, S_n$ have no common factor. Again by comparing the other weights, we obtain $R_mS_{n-1} + R_{m-1}S_n=0$. Since $R_m, S_n$ have no common factor, we conclude that $R_{m-1}=0, S_{n-1}=0$. By induction, we get $R_i=0$ for $i<m$ and $S_j=0$ for $j<n$. This implies that $R_0S_0=0$, which is a contradiction. This proves the claim.
\end{proof}
We have a corollary.

\begin{cor}
\label{congruence of weights in level T}
For $i = 1, 2$, let $f_i\in M_{k_i,0}(\Gamma_0(T))$ be a Drinfeld modular form with $\p$-integral $u$-expansion at $\infty$ such that $f_i\not \equiv 0 \pmod \p$. If $f_1\equiv f_2 \pmod \p$ then $k_1\equiv k_2 \pmod {q^d-1}$. 
\end{cor}

\begin{proof}
Let $k_1>k_2$. For $i=1,2$, let $\varphi_{f_i}(U,V)$ be the unique isobaric polynomial attached to $f_i$. 
By Theorem~\ref{MT_3}, the condition $f_1\equiv f_2 \pmod \p$ implies that $\overline{\varphi}_{f_1}(U,V)-\overline{\varphi}_{f_2}(U,V)\in (\overline{\phi}_d(U,V)-1)$. Thus
$\overline{\varphi}_{f_1}(U,V)-\overline{\varphi}_{f_2}(U,V) = R (\overline{\phi}_d(U,V)-1)$
 for some $R=  \sum_{i=0}^n R_{m_i}\in \F_\p[U,V]$, where each $R_{m_i}$ is a non-zero isobaric polynomial of weight $m_i$ with $m_i<m_{i+1}$. By comparing the terms of same weights on both sides, we get   
\begin{equation}
\label{lower filtration equation}
\overline{\varphi}_{f_1}(U,V)= R_{m_n}\overline{\phi}_d(U,V).
\end{equation}
This implies $\overline{\phi}_d(U,V) \mid \overline{\varphi}_{f_1}(U,V)$. Comparing the remaining weights, we get
\begin{equation}
\label{equation or other graded part}
\overline{\varphi}_{f_2}(U,V)= R_{m_0}\ \mrm{and}\ R_{m_i}\overline{\phi}_d(U,V)= R_{m_{i+1}} \ \mrm{for} \ 0\leq i \leq n-1.
\end{equation} 
Equating the weights in~\eqref{lower filtration equation} and~\eqref{equation or other graded part}, we get
\begin{equation*}
m_n=k_1-(q^d-1), \ m_{i+1} = m_i +(q^d-1) \ \mrm{for} \ 0\leq i \leq n-1, \ \mrm{and} \ m_0= k_2.
\end{equation*}
Now, it is easy to see that $(q^d-1)|(k_1-k_2)$. This proves the corollary.
\end{proof}
A more general version of the above corollary can be found in the work of Hattori (cf.~\cite[Theorem 4.16]{Hat20}).

\section{Weight filtration for $f  \in M_{k,l}(\Gamma_0(T))$}
\label{Section_Weight_Filtration}
In this section, we define the weight filtration for $f \in M_{k,l}(\Gamma_0(T))$ and study its properties.
Let $M_k(\Gamma_0(T))$ denote the space of Drinfeld modular forms of weight $k$ (any type) for $\Gamma_0(T)$.
\begin{dfn}
\label{filtration for any type for Gamma0T}
Let $f\in M_{k,l}(\Gamma_0(T))$ be a Drinfeld modular form with $\mfp$-integral $u$-expansion at $\infty$ such that 
$f \not \equiv 0 \pmod \p$. We define the weight filtration of $f$, which we denote by $w(\overline{f})$, as 
$$w(\overline{f}) := \inf \{k_0 | \exists  \ g\in M_{k_0}(\Gamma_0(T)) \ \mathrm{with} \ g \equiv f \pmod {\p} \}.$$
By Lemma~\ref{Congruence between forms of arbitrary type resuces to congruence between trivial type}, we get $g \in M_{k_0,l}(\Gamma_0(T))$.
If $f\equiv 0 \pmod \p$, then $w(\overline{f}): = -\infty$.
\end{dfn}
%
%
  
In the following theorem, we show that  the properties of the weight filtration of $f\in M_{k,l}(\Gamma_0(T))$ can be deduced from the relevant properties of $c_T(f)$.

\begin{thm}
\label{Lower_Filtration}
Let $f\in M_{k,l}(\Gamma_0(T))$ be a Drinfeld modular form with $\p$-integral $u$-expansion at $\infty$ and $f \not\equiv 0 \pmod \p$.
Let $\varphi_f(U,V)$ be the unique isobaric polynomial attached to $c_T(f)$.
Then  
\begin{enumerate}
\item $w(\overline{f}) = w(\overline{{c_T(f)}})+2l,$
\item  $w(\overline{f})\equiv k \pmod {q^d-1}$,
\item $w(\overline{f})<k$ if and only if $\overline{\phi}_d(U,V)\mid \overline{\varphi}_f(U,V)$.
\end{enumerate}
\end{thm}

\begin{proof}
Recall that $l$ denotes a representative of $l\in \Z/ (q-1)\Z$ with $0 \leq l \leq q-2$.
The proof of the theorem is divided into two parts, i. e., $k=2l$ and $k > 2l$ (note that $\dim M_{k,l}(\Gamma_0(T))=0$ if $k<2l$). 

If $k=2l$, then $f=c\cdot E_T^l$ with $c \not \equiv 0 \pmod \p.$ 
We show that $w(\overline{E_T^l})= 2l$. 
If  $w(\overline{E_T^l})=e< 2l$, then $E_T^l \equiv g \pmod \p$ for some $g\in M_{e,l}(\Gamma_0(T))$. Since $e<2l$, 
we get that $g=0$ (cf. Proposition~\ref{Dimension Gamma0(T)}). This is a contradiction. 
This implies part(1), part(2) trivially, and part(3) is vacuously true.

We now assume $k >2l$.  If $k^{\prime} = w(\overline{f})$ then there exists $g\in M_{k^{\prime},l}(\Gamma_0(T))$ such that $f\equiv g \pmod \p$. 
By definition, $c_T(f)\in M_{k-2l,0}(\Gamma_0(T))$, $c_T(g) \in M_{k^{\prime}-2l,0}(\Gamma_0(T))$.
By Lemma~\ref{Congruence between forms of arbitrary type resuces to congruence between trivial type},
the congruence $f\equiv g \pmod\p$ implies that $c_T(f)\equiv c_T(g) \pmod \p.$
Thus, $w(\overline{c_T(f)}) \leq k^{\prime}-2l$. 

We now show that $w(\overline{c_T(f)}) = k^{\prime}-2l$.
If $w(\overline{c_T(f)})<k^\prime-2l$, then there exists 
$f_1\in M_{w(\overline{c_T(f)}),0}(\Gamma_0(T))$  such that $c_T(f)\equiv f_1 \pmod \p$. 
This implies that $f\equiv f_1 E_T^l \pmod \p$. Consequently, $w(\overline{f})\leq w(\overline{c_T(f)})+2l< k^\prime=w(\overline{f})$, which is a contradiction. Hence, we have $w(\overline{c_T(f)})=w(\overline{f})-2l$. This proves part(1) of the theorem.  

By Corollary~\ref{congruence of weights in level T} we have $w(\overline{c_T(f)})\equiv k-2l \pmod {q^d-1}$. Now, 
part(2) follows from part(1). 

Now, we prove part(3). 
If $w(\overline{f})<k$, then there exists $g\in M_{w(\overline{f}),l}(\Gamma_0(T))$ such that $f\equiv g \pmod \p$
and  hence $c_T(f)\equiv c_T(g) \pmod \p$ (cf. Lemma~\ref{Congruence between forms of arbitrary type resuces to congruence between trivial type}). The equation~\eqref{lower filtration equation} in the proof of Corollary \ref{congruence of weights in level T} implies that $\overline{\phi}_d(U,V) \mid \overline{\varphi}_f(U,V)$. This proves the forward implication. 

For the reverse implication, if $\overline{\phi}_d(U,V)\mid \overline{\varphi}_f(U,V)$ then there exists 
$\overline{\psi}(U,V)\in \F_\p[U,V]$ 
such that $\overline{\varphi}_f(U,V)=\overline{\phi}_d(U,V)\overline{\psi}(U,V)$. 
Since  $\overline{\varphi}_f(U,V)$, $\overline{\phi}_d(U,V)$ are isobaric,  
$\overline{\psi}(U,V)$ is also isobaric with weight $(k-2l)-(q^d-1)$. 
Let $\psi(U,V)\in A_\p[U,V]$ be an isobaric lift of $\overline{\psi}(U,V)$ with same weight. 
Then $g := \psi(\Delta_W, \Delta_T)$ is a Drinfeld modular form of weight $(k-2l)-(q^d-1)$. 
Since $\overline{\varphi}_f(U,V)=\overline{\phi}_d(U,V)\overline{\psi}(U,V)$, we get that
$\overline{\varphi}_f(\Delta_W, \Delta_T)=\overline{\phi}_d(\Delta_W, \Delta_T)\overline{\psi}(\Delta_W, \Delta_T)$.
The congruence $g_d \equiv 1 \pmod \p$ implies that $c_T(f)\equiv g \pmod \p$ and
$f\equiv g E_T^l \pmod \p$. Hence, we have $w(\overline{f})<k$. 
\end{proof}
\section{An application to mod-$\p$ congruences for $\p \neq (T)$}
\label{Section_ModP-Congruences}
In this section, as an application of Theorem~\ref{MT_2} and Theorem~\ref{Lower_Filtration}, 
we prove a certain relation between mod-$\p$ congruences for Drinfeld modular forms $f,g$ of level $\Gamma_0(\p T)$ and signs of the Fricke involution on $f,g$.
 
For classical modular forms,  Calegari and Stein  (cf.~\cite[Conjecture $4$]{CS04}) conjectured that the eigenvalues of Fricke involution on $f$ and $g$ have opposite signs if there is a mod $p$ congruence between the cusp form $g$ of weight $4$ and derivative of the cusp form $f$ of weight $2$ for $\Gamma_0(p)$.
In~\cite{DK}, we gave a proof of  this conjecture for Drinfeld modular forms $f$ of weight $k$, 
type $l$ of level $\Gamma_0(\p \m)$, under some assumptions on the weight filtration of certain modular form associated to $f$. For $\m=(1)$, we appealed to  the structure theorem for Drinfeld modular forms of level $\GL_2(A)$
to prove~\cite[Theorem 1.3]{DK}. 
For general $\m$, a similar structure theorem for level $\Gamma_0(\m)$ 
is not available   to prove~\cite[Theorem 1.5]{DK}. 

As an application of the main results of this article, we adopt the same methodology of $\GL_2(A)$ for $\Gamma_0(T)$ to prove~\cite[Theorem 1.5]{DK} for 
level $\Gamma_0(\p T)$. Recall that $\mfp$ denotes a prime ideal  of $A$ generated by a monic irreducible polynomial $\pi$ of degree $d$
such that   $\mfp \neq (T)$.
Before stating the main result of this section we introduce some important operators.
 \begin{dfn}[$\Theta$-operator]
 For Drinfeld modular forms, there is an analogue of Ramanujan's $\Theta$-operator, which is defined as
 $$\Theta := \frac{1}{\tilde{\pi}}\frac{d}{dz}= -u^2\frac{d}{du}.$$ 
 \end{dfn}
 The $\Theta$-operator preserves quasi-modularity but not  modularity.
 However, one can  perturb the $\Theta$-operator to define the $\partial_k$-operator which preserves modularity.

 \begin{dfn}
 \label{Definition_2nd Derivative}
 For $k \in \N$ and $l \in \Z/(q-1)\Z$, we define  
 	$\partial_k:M_{k,l}(\Gamma_0(\n))\ra M_{k+2,l+1}(\Gamma_0(\n))$ by
 	\begin{equation}
  	\label{2nd Derivative}
 	\partial_k f := \Theta f + kEf.
 	\end{equation}
 We shall write $\partial$ for $\partial_k$ if the weight $k$ is clear from the context.
 	\end{dfn}

\begin{dfn}
The (partial) Atkin-Lehner operator is defined by the action of the matrix $W_\p^{(\p T)} := \psmat{\pi}{b}{\pi T}{d\pi}$ with $b,d\in A$ such that $d\pi^2-b\pi T=\pi$. For more details, we refer the reader to~\cite[\S 3.1]{DK}. 
\end{dfn}
We now state the main result of this section. 
\begin{thm}
\label{Drinfeld_Congruence_Thm_1}
Suppose $f\in M_{k,l}^1 (\Gamma_0(\mfp T))$, 
$g\in M^1_{k+2,l+1}(\Gamma_0(\mfp T ))$  have $\p$-integral $u$-expansion at $\infty$
and satisfy $\Theta f \equiv g \pmod {\p}$.
Further, suppose that $f|W_\mfp^{(\mfp T)} = \alpha f$, $g|W_\mfp^{(\mfp T)}= \beta g$ with $\alpha,\beta \in \{ \pm 1 \}$ and $(k,p)=1$.
If $w(\overline{F})= (k-1)(q^d-1)+k$, where $F\in M_{(k-1)(q^d-1)+k,l}^1(\Gamma_0(T))$ as in ~\cite[Proposition 3.8]{DK} corresponding to $kf$, then $\beta = -\alpha$.
\end{thm}

For $\m=(T)$ in~\cite[Theorem 1.5]{DK}, it may appear that Theorem~\ref{Drinfeld_Congruence_Thm_1} is a special 
case, but there is a difference in  definitions of weight filtration $w(\overline{F})$ in 
the respective theorems. We urge the reader 
to look at Definition~\ref{filtration for any type for Gamma0T} and~\cite[Definition 5.5]{DK}
to note the difference. 
There the weight filtration was defined using the geometry of Drinfeld modular curves.
Since the level of $F$ is $\Gamma_0(T)$, we expect that these two definitions are same.
However, we are unable to prove this. 

The following proposition is a generalization of~\cite[Theorem 3.1(1)]{Vin10}
and is quite useful in the proof of Theorem~\ref{Drinfeld_Congruence_Thm_1}.



\begin{prop}
\label{no common factor}
Let $\psi_d(U,V) \in A[U,V] $ be the unique isobaric polynomial attached to $c_T(\partial(g_d))$.
Then $\overline{\psi}_d(U,V)$, $\overline{\phi}_d(U,V)$ share no  common factor.
\end{prop}


For $f\in M_{k,l}(\Gamma_0(T))_{A_\p}$, let  $\varphi_{f}(U,V) \in A_\p[U,V]$ be the unique polynomial attached to $c_T(f)$ such that $f=\varphi_{f}(\Delta_W,\Delta_T)E_T^l$ (cf. Proposition~\ref{Isobaric polynomial for arbitrary type}). Define a derivation $\partial$ on $A_\p[U,V,Z]$ (resp., on $\F_\p[U,V,Z]$),
by setting $\partial(U) := -UZ, \partial(V) :=0, \partial(Z) := Z^2$.
By Proposition~\ref{IMP_Prop}(\ref{P6}), we get $\vartheta(A_\p)(\partial(\varphi_{f}(U,V)Z^l))= \partial_k(f)$ , where $\vartheta$ is the mapping defined in Theorem~\ref{MT_1}.

 \begin{proof}[Proof of Proposition~\ref{no common factor}]
First, we note that the polynomials $\partial(\phi_d(U,V))$, $\psi_d(U,V)Z$ are equal.
This follows by Proposition~\ref{Isobaric polynomial for arbitrary type} and
$\partial(\phi_d(\Delta_W, \Delta_T))= \psi_d(\Delta_W, \Delta_T)E_T$. 

If possible, let there be a common factor between $\overline{\psi}_d(U,V), \overline{\phi}_d(U,V)$, say $a$.
Since $\overline{\phi}_d(U,V)$ is square-free, $\overline{\phi}_d(U,V) = ab$ for some $b$ with $(a,b)=1$.
By Remark~\ref{U,V does not divide phi_d remark}, the variables $U$ and $V$ do not divide $\overline{\phi}_d(U,V)$.
Hence, we write
\begin{equation}
\label{a is isobaric}
a= U^k + \sum_{i=1}^{k-1} \overline{c}_{i} U^{k-i}V^{i}+ \overline{c}_k V^k
\end{equation}
with $\overline{c}_k \neq 0$.
This implies $\partial(a)=\eta(U,V)Z$, where
\begin{equation*}
\eta(U,V)=(-kU^k+\sum_{i=1}^{k-1} -(k-i)\overline{c}_{i} U^{k-i} V^{i}).
\end{equation*}
Since $\overline{\psi}_d(U,V)Z = \partial(a)b + a\partial(b)$, $a \mid \overline{\psi}_d(U,V)$ implies $a \mid\partial(a)$ and hence $a \mid \eta(U,V)$.
 Considering the highest power of $V$ in $a$ and  $\eta(U,V)$, we see that $a \mid \eta(U,V)$ 
 can happen only if  $\eta(U,V)=0$. This forces that $k\equiv 0 \pmod p$ and
$(k-i)\overline{c}_{i}=0$ for all $1 \leq i \leq k-1$. Hence, either $\overline{c}_{i}=0$ or $i \equiv 0 \pmod p$ for all $1 \leq i \leq k-1$.
By~\eqref{a is isobaric}, the polynomial $a$ becomes a $p$-th power. This is a contradiction to $\overline{\phi}_d(U,V)$ being square-free.  This proves the proposition.
\end{proof}

The proof of Theorem~\ref{Drinfeld_Congruence_Thm_1} is similar to the proof of~\cite[Theorem 1.5]{DK}. 
Here, we use the structure theorem for $M^0(\Gamma_0(T))_R$ (i. e., Theorem~\ref{MT_2}).

\subsection{Proof of Theorem~\ref{Drinfeld_Congruence_Thm_1}:}
We shall prove this theorem by contradiction. Suppose $\beta = \alpha$.
Let $E_\p(z)= E(z)-\pi E(\pi z)\in M_{2,1}(\Gamma_0(\p))$. We have the congruence $E_\p \equiv E \pmod \p$. Since $\Theta f \equiv g \pmod \p$, \eqref{2nd Derivative} gives
\begin{equation*}
\partial f \equiv g + kE_\p f \pmod {\p}.
\end{equation*}
Hence, there exists $h_1\in M_{k+2,l+1}^1(\Gamma_0(\p T))$ with $v_\p(h_1)\geq 0$ such that 
\begin{equation}
\label{h exists}
 g- \partial f + kE_\p f = \pi h_1.
\end{equation}
Applying $W_\p^{(\p T)}$ on both sides, we obtain 
$\alpha(g-\partial f) = \pi h_1|_{k+2,l+1}W_\p^{(\p T)}$
(cf. ~\cite[Proposition 3.4]{DK} for more details).
Combining this with~$\eqref{h exists}$, we get 
$$kE_\p f = \pi h_1 - \alpha \pi h_1|_{k+2,l+1}W_\p^{(\p T)}$$
 and hence $kE_\p f \equiv  - \alpha \pi h_1|_{k+2,l+1}W_\p^{(\p T)} \pmod \p.$
By~\cite[Proposition 3.8]{DK}, there exists $F\in M_{(k-1)q^d+1,l}^1(\Gamma_0(T))$
with $\p$-integral $u$-expansion at $\infty$
such that $kf \equiv F \pmod \p$. Hence 
\begin{equation}
\label{before the final congruence}
E_\p F\equiv - \alpha \pi h_1|_{k+2,l+1}W_\p^{(\p T)} \pmod \p.
\end{equation}
By~\cite[Proposition 3.9]{DK}, there exists $H \in M_{(k-1)q^d+3,l+1}^1(\Gamma_0(T))$  such that $$H\equiv \alpha \pi h_1|_{k+2,l+1}W_\p^{(\p T)} \pmod \p.$$
Therefore~\eqref{before the final congruence} becomes
$
 H \equiv -E_\p F  \pmod \p.
$
Since $E_\p \equiv E \pmod \mfp$ and $E\equiv \partial(g_d) \pmod \p$(cf. \cite[Theorem 1.1]{Vin10}), we get
\begin{equation}
\label{pT-final congruence}
H\equiv \partial(g_d)F \pmod \p.
\end{equation}
The last congruence implies that $H$ has a $\p$-integral $u$-expansion at $\infty$.
Since both sides of~\eqref{pT-final congruence} are congruent mod-$\p$, we get 
\begin{equation}
\label{weight filtrations are equal for mod p congruences}
w(\overline{H}) = w(\overline{\partial(g_d)F}).
\end{equation}
We now calculate the weight filtration for $\partial(g_d)F$.
Note the weights of $F$, $\partial(g_d)F$ are $(k-1)(q^d-1)+k$, $kq^d+2$  respectively.
%

Let $\varphi_{F}(U,V)$ (resp., $\psi_d(U,V)$)  be the unique isobaric polynomial attached to $c_T(F)$ (resp., to $c_T(\partial(g_d))$) such that $F= \varphi_{F}(\Delta_W, \Delta_T)E_T^l$ (resp., $\partial(g_d)= \psi_d({\Delta_W, \Delta_T})E_T$). 
Hence, $\partial(g_d)F=\varphi_{F}(\Delta_W, \Delta_T)\psi_d({\Delta_W, \Delta_T})E_T^{l+1}$ and
\begin{equation*}
c_T(\partial(g_d)F)=
\begin{cases}
\frac{\partial(g_d)F}{E_T^{l+1}} \qquad  \qquad \quad \ \mathrm{if} \ 0\leq l<q-2,\\
\frac{\partial(g_d)F}{E_T^{l+1}}   \Delta_W\Delta_T\qquad \ \mathrm{if} \ l=q-2.
\end{cases}
\end{equation*}
Let $\eta(U,V)$ be the unique isobaric polynomial attached to $c_T(\partial(g_d)F)$. Then
\begin{equation*}
\eta(U,V)=
\begin{cases}
\varphi_{F}(U,V) \psi_d(U,V)  \qquad\ \ \ \mathrm{if} \ 0\leq l<q-2,\\
\varphi_{F}(U,V) \psi_d(U,V)UV \quad \ \mathrm{if} \ l=q-2.
\end{cases}
\end{equation*}

By Theorem~\ref{Lower_Filtration}(3), the assumption $w(\overline{F})= (k-1)(q^d-1)+k$ implies $\overline{\phi}_d(U,V) \nmid \overline{\varphi}_{F}(U,V)$. By Proposition~\ref{no common factor} and Remark~\ref{U,V does not divide phi_d remark},
we get $\overline{\phi}_d(U,V) \nmid \overline{\eta}(U,V)$. This is true because the polynomials $\overline{\phi}_d(U,V)$, $\overline{\psi}_d(U,V)$ share no common factor and  neither $U$ nor $V$ divide $\overline{\phi}_d(U,V)$.
By Theorem~\ref{Lower_Filtration}(3), we get  $w(\overline{\partial(g_d)F})= kq^d+2$.
Since the weight of $H$ is $(k-1)q^d+3< kq^d+2$, we conclude that $$w(\overline{H})\leq (k-1)q^d+3 < kq^d+2 = w(\overline{\partial(g_d)F}),$$ which contradicts~\eqref{weight filtrations are equal for mod p congruences}.
This completes the proof.
  

Finally, we note that the condition on weight filtration $w(\overline{F})$ in Theorem~\ref{Drinfeld_Congruence_Thm_1} is automatically satisfied for doubly cuspidal forms of weight $2$, type $1$. 
\begin{cor}
Let $\mfp$ be a prime ideal of $A$ such that $\mfp \neq (T)$. Suppose $f\in M_{2,1}^2 (\Gamma_0(\mfp T))$ and
$g\in M^2_{4,2}(\Gamma_0(\mfp T ))$ have $\p$-integral $u$-expansion at $\infty$ and satisfy $\Theta f \equiv g \pmod {\p}$. Further, suppose that $f|W_\mfp^{(\mfp T)} = \alpha f$,
$g|W_\mfp^{(\mfp T)}= \beta g$ with $\alpha,\beta \in \{ \pm 1 \}$.
If $f \not \equiv 0 \pmod \p$,
then $\beta = -\alpha$.
\end{cor}
\begin{proof}
Arguing as in the proof of Theorem~\ref{Drinfeld_Congruence_Thm_1} for $f\in M_{2,1}^2 (\Gamma_0(\mfp T))$,  we get 
$F\in M_{q^d+1,1}^2(\Gamma_0(T))$.
Since $f \not \equiv 0 \pmod \p$,
the congruence $w(\overline{F})\equiv q^d+1 \pmod {q^d-1}$(cf. Theorem~\ref{Lower_Filtration}(2))
implies that $w(\overline{F})= 2$ or $q^d+1$. If $w(\overline{F})= q^d+1$, then  the corollary follows from Theorem~\ref{Drinfeld_Congruence_Thm_1}. We now show that $w(\overline{F})$ cannot be $2$.
 
Suppose $w(\overline{F}) = 2$. Since $M_{2,1}(\Gamma_0(T)) = \langle E_T \rangle$,
we obtain  $F \equiv c E_T \pmod \p$ 
for some $c$. Comparing the coefficient of $u$ on both sides of the congruence, 
we get $c \equiv 0 \pmod \p$. Thus $F \equiv 0 \pmod \p$, which is a contradiction. 
\end{proof}

\section*{Acknowledgments}  
 The first author thanks University Grants Commission (UGC), India for the financial support provided in the form of 
 Research Fellowship to carry out this research work at IIT Hyderabad. 
 The second author's research was partially supported by the SERB grant MTR/2019/000137.

\bibliographystyle{plain, abbrv}

\end{document}